\documentclass[12pt,a4paper]{article}
\usepackage[utf8x]{inputenc}
\usepackage{amsmath}
\usepackage{amsfonts}
\usepackage{amssymb}
\usepackage{amsthm}
\usepackage{mathrsfs}
\usepackage{fancyhdr}
\usepackage{graphicx}
\usepackage[usenames,x11names]{xcolor}
\usepackage{arabtex}
\usepackage{framed}
\usepackage[top=2cm,bottom=2cm,margin=2cm]{geometry}
\usepackage{cancel}
\usepackage{array}   %%%% Pour les fonctionalités assez évoluées d'un tableau

\usepackage[dvipdfm,pdfstartview=FitH,colorlinks=true,linkcolor=Red4,urlcolor=Red4,%
filecolor=green,citecolor=red]{hyperref}

\title{Results and conjectures related to a conjecture of \erdos{} concerning
primitive sequences}
\author{\sc Bakir FARHI \\
Laboratoire de Mathématiques appliquées \\
Faculté des Sciences Exactes \\
Université de Bejaia, 06000 Bejaia, Algeria \\[1mm]
\href{mailto:bakir.farhi@gmail.com}{bakir.farhi@gmail.com} \\[1mm]
\url{http://www.bakir-farhi.site}
}
\date{}

\let\up=\textsuperscript

\def\R{{\mathbb R}}

\def\N{{\mathbb N}}
\def\Z{{\mathbb Z}}

\def\A{\mathscr{A}}

\def\P{\mathscr{P}}
\def\I{\mathscr{I}}
\def\erdos{Erd\H{o}s}

\newcommand{\card}[1]{\mathrm{Card}\,#1}

\def\EMdash{\leavevmode\hbox to 10.6mm{\vrule height .63ex depth -.59ex
    width 10mm\hfill}}

\theoremstyle{plain}
\numberwithin{equation}{section}
\newtheorem{thm}{Theorem}[section]

\newtheorem{lemma}[thm]{Lemma}

\newtheorem{rmq}[thm]{Remark}
\newtheorem{prop}[thm]{Proposition}

\newtheorem{conj}[thm]{Conjecture}

\begin{document}
\maketitle
\begin{abstract}
A strictly increasing sequence $\A$ of positive integers is said to be primitive if no term of $\A$ divides any other. \erdos{} showed that the series $\sum_{a \in \A}
\frac{1}{a \log a}$, where $\A$ is a primitive sequence different from $\{1\}$, are all convergent and their sums are bounded above by an absolute constant. Besides, he
conjectured that the upper bound of the preceding sums is reached when $\A$ is
the sequence of the prime numbers. The purpose of this paper is to study the
\erdos{} conjecture. In the first part of the paper, we give two significant
conjectures which are equivalent to that of \erdos{} and in the second one, we
study the series of the form $\sum_{a \in \A} \frac{1}{a (\log a + x)}$, where $x$ is a fixed non-negative real number and $\A$ is a primitive sequence different from $\{1\}$. In particular, we prove that the analogue of \erdos's conjecture for these series does not hold, at least for $x \geq 363$. At the end of the paper, we propose a more general conjecture than that of \erdos{}, which concerns the preceding series, and we conclude by raising some open questions. 
\end{abstract}
\noindent\textbf{MSC 2010:} Primary 11Bxx. \\
\textbf{Keywords:} Primitive sequences, \erdos's conjecture, prime numbers, sequences of integers.

\section{Introduction}
Throughout this paper, we let $\lfloor x\rfloor$ denote the integer part of a real number $x$ and we let $\card S$ denote the cardinality of a set $S$. Further, we denote by $\P = {(p_n)}_{n \geq 1}$ the sequence of the prime numbers. For a given sequence of positive integers $\A$, we denote by $P(\A)$ the set of the prime divisors of the terms of $\A$, that is
$$
P(\A) := \{p \in \P |~ \exists a \in \A , p | a\} .
$$
For a given positive integer $n$, we denote by $\Omega(n)$ the
number of prime factors of $n$ counted with multiplicity. For a given sequence of
positive integers $\A$, the quantity defined by $d°(\A) :=
\max\{\Omega(a) , a \in \A\}$ is called \emph{the degree} of $\A$. Particularly, if $\Omega(a)$ is the same for any $a \in \A$, then $\A$ is called \emph{an homogeneous sequence}.

A sequence $\A$ of positive integers is called \emph{primitive} if it is strictly increasing and satisfies the property that no term of $\A$ divides any other. A particular and interesting class of primitive sequences is the class of homogeneous sequences. In \cite{erd1}, \erdos{} proved that for any infinite primitive
sequence $\A$ (with $\A \neq \{1\}$), the series 
$$
\sum_{a \in \A} \frac{1}{a \log{a}}
$$
converges and its sum is bounded above by an absolute constant $C$. In \cite{ez1},
\erdos{} and Zhang showed that $C \leq 1.84$ and in \cite{cla}, Clark improved
this estimate to $C \leq e^{\gamma} \simeq 1.78$ (where $\gamma$ denotes the Euler
constant). Furthermore, in \cite{erd2}, \erdos{} asked if it is true that the sum
$\sum_{a \in \A} \frac{1}{a \log a}$ (where $\A \neq \{1\}$ is a primitive sequence)
reaches its maximum value at $\A = \P$. Some years later, \erdos{} and Zhang \cite{ez1}
conjectured an affirmative answer to the last question by proposing the
following

\medskip

\noindent\textbf{Conjecture 1 (\erdos{}):}~\\
\emph{%
For any primitive sequence $\A \neq \{1\}$, we have:
$$
\sum_{a \in \A} \frac{1}{a \log a} \leq \sum_{p \in \P} \frac{1}{p \log p} .
$$
}

\medskip

\noindent To compare with Clark's upper bound, we specify that
$\sum_{p \in \P} \frac{1}{p \log p} \simeq 1.63$. In their same paper \cite{ez1},
\erdos{} and Zhang showed that the above conjecture is equivalent to the following
which deals with finite sums:

\medskip

\noindent\textbf{Conjecture 2 (\erdos{} and Zhang \cite{ez1}):}~\\
For any primitive sequence $\A \neq \{1\}$ and any positive integer $n$, we have: 
$$
\sum_{\begin{subarray}{c}
a \in \A \\
a \leq n
\end{subarray}
} \frac{1}{a \log a} \leq \sum_{\begin{subarray}{c}
p \in \P \\
p \leq n
\end{subarray}
} \frac{1}{p \log p} .
$$
In \cite{z1}, Zhang proved the \erdos{} conjecture for a primitive sequence $\A$
($\A \neq \{1\}$) satisfying $d°(\A) \leq 4$ and in \cite{z}, he proved it for the
particular case of homogeneous sequences and for some other slightly more complicated primitive sequences. To our knowledge, these are the only significant results that were obtained in the direction of proving \erdos{}'s conjecture. 

To know more about the primitive sequences, the reader can consult the excellent
book of Halberstam and Roth \cite[Chapter 5]{hr}.

The main purpose of this article is to study Conjecture 1 of \erdos{}.
In the first part, we just give two significant conjectures which are equivalent
to Conjecture 1. In the second part, we study the series of the form
$S(\A , x) := \sum_{a \in \A} \frac{1}{a (\log a + x)}$, where $\A$ is a primitive
sequence (different from $\{1\}$) and $x$ is a non-negative real number. In this context, we can formulate Conjecture 1 simply by the inequality: $\sup_{\A} S(\A , 0) \leq S(\P , 0)$ (where the supremum is taken over all primitive sequences $\A \neq \{1\}$). So, by analogy, we can naturally ask, for a given $x \in \R^+$, whether it is
true that $\sup_{\A} S(\A , x) \leq S_(\P , x)$. As a corollary of a more general result, we show that the last inequality is wrong for any $x \geq x_0$, where $x_0$ is an effectively calculable non-negative real number. We show that $x_0 = 363$
is suitable but the determination of the best value (i.e., the minimal value) of
$x_0$ is left as an open problem. Obviously, if $x_0 = 0$ is also suitable then
\erdos{}'s conjecture is false. We end the paper by proposing a conjecture about the quantity 
$\sup_{\A} \sum_{a \in \A} \frac{1}{a (\log a + x)}$ (where $x$ is a fixed non-negative real number and the supremum is taken over all primitive sequences $\A \neq \{1\}$) which generalizes the \erdos{} conjecture and then by raising some open questions.
 
\section{Two conjectures equivalent to \erdos{}'s conjecture}
In this section, we propose two new significant conjectures and we show just after
that both are equivalent to Conjecture 1.
\begin{conj}\label{conj1}
For any primitive sequence $\A$, with $\A \neq \{1\}$, we have:
$$
\sum_{a \in \A} \frac{1}{a \log{a}} \leq \sum_{p \in P(\A)} \frac{1}{p \log{p}} .
$$
\end{conj}

\begin{conj}\label{conj2}
For any primitive sequence $\A$, with $\A \neq \{1\}$, we have:
$$
\sum_{a \in \A} \frac{1}{a \log{a}} \leq \sum_{n = 1}^{\card \A}
\frac{1}{p_n \log{p_n}}
$$
\end{conj}

We have the following proposition:

\begin{prop}\label{p1}
Both Conjectures \ref{conj1} and \ref{conj2} are equivalent to Conjecture 1.
\end{prop}

\begin{proof}
It is obvious that each of Conjectures \ref{conj1} and \ref{conj2} is stronger than
Conjecture 1. So it remains to show that Conjecture 1 implies Conjecture
\ref{conj1} and that Conjecture 1 implies Conjecture \ref{conj2}. \\
\textbullet{} Let us show that Conjecture 1 implies Conjecture \ref{conj1}.
Assume Conjecture 1 is true and show Conjecture \ref{conj1}. So, let $\A \neq \{1\}$
be a primitive sequence. Then, clearly $\A' := \A \cup (\P \setminus P(\A))$ is
also a primitive sequence. Thus, according to Conjecture 1 (supposed true), we have:
$$
\sum_{a \in \A'} \frac{1}{a \log a} \leq \sum_{p \in \P} \frac{1}{p \log p} .
$$ 
But since
$$
\sum_{a \in \A'} \frac{1}{a \log a} = \sum_{a \in \A} \frac{1}{a \log a} +
\sum_{p \in \P \setminus P(A)} \frac{1}{p \log p} = \sum_{a \in \A}
\frac{1}{a \log a} + \sum_{p \in \P} \frac{1}{p \log p} - \sum_{p \in P(\A)}
\frac{1}{p \log p} ,
$$
we get (after simplifying):
$$
\sum_{a \in \A} \frac{1}{a \log a} \leq \sum_{p \in P(\A)} \frac{1}{p \log p} ,
$$
as required by Conjecture \ref{conj1}. \\
\textbullet{} Now, let us show that Conjecture 1 implies Conjecture
\ref{conj2}. Assume Conjecture 1 is true and show Conjecture \ref{conj2}.
So, let $\A \neq \{1\}$ be a primitive sequence. Because if $\A$ is infinite,
Conjecture \ref{conj2} is exactly the same as Conjecture 1, we can suppose that
$\A$ is finite. Then, to prove the inequality of Conjecture \ref{conj2}, we argue
by induction on $\card \A$. \\
--- For $\card \A = 1$: Since $A \neq \{1\}$, we have $\A = \{a_1\}$ for some
positive integer $a_1 \geq 2$. So, we have:
$$
\sum_{a \in \A} \frac{1}{a \log a} = \frac{1}{a_1 \log a_1} \leq \frac{1}{2 \log 2}
= \frac{1}{p_1 \log p_1} = \sum_{n = 1}^{\card \A} \frac{1}{p_n \log p_n} ,
$$    
confirming the inequality of Conjecture \ref{conj2} for this case. \\
--- Let $N$ be a positive integer. Suppose that Conjecture \ref{conj2} is true for
any primitive sequence ($\neq \{1\}$) of cardinality $N$ and show that it remains also
true for any primitive sequence of cardinality $(N + 1)$. So, let $\A = \{a_1 ,
\dots , a_N , a_{N + 1}\}$, with $a_1 < a_2 < \dots < a_N < a_{N + 1}$, be a
primitive sequence of cardinality $(N + 1)$ and let us show the inequality of
Conjecture \ref{conj2} for $\A$. To do so, we introduce $\A'' := \{a_1 , \dots ,
a_N\}$, which is obviously a primitive sequence of cardinality $N$, and we distinguish
the two following cases: \\
\textbf{1\up{st} case:} (if $a_{N + 1} \geq p_{N + 1}$) \\
In this case, we have on the one hand:
\begin{equation}\label{eq1}
\frac{1}{a_{N + 1} \log a_{N + 1}} \leq \frac{1}{p_{N + 1} \log p_{N + 1}} 
\end{equation}     
and on the other hand, according to the induction hypothesis applied for $\A''$:
\begin{equation}\label{eq2}
\sum_{n = 1}^{N} \frac{1}{a_n \log a_n} \leq \sum_{n = 1}^{N} \frac{1}{p_n \log p_n}
\end{equation}
By adding \eqref{eq1} and \eqref{eq2}, we get
$$
\sum_{n = 1}^{N + 1} \frac{1}{a_n \log a_n} \leq \sum_{n = 1}^{N + 1} \frac{1}{p_n
\log p_n} ,
$$
which shows the inequality of Conjecture \ref{conj2} for $\A$. \\
\textbf{2\up{nd} case:} (if $a_{N + 1} < p_{N + 1}$) \\
In this case, we have $a_1 < a_2 < \dots < a_{N + 1} < p_{N + 1}$, implying that
$$
P(\A) \subset \{p_1 , p_2 , \dots , p_N\} .
$$ 
It follows by applying Conjecture \ref{conj1} for $\A$ (which is true by hypothesis, since we have assumed that Conjecture 1 is true and we have shown above that
Conjecture 1 implies Conjecture \ref{conj1}) that
$$
\sum_{i = 1}^{N + 1} \frac{1}{a_i \log a_i} \leq \sum_{p \in P(\A)}
\frac{1}{p \log p} \leq \sum_{n = 1}^{N} \frac{1}{p_n \log p_n} \leq
\sum_{n = 1}^{N + 1} \frac{1}{p_n \log p_n} , 
$$ 
showing the inequality of Conjecture \ref{conj2} for $\A$. This achieves this
induction and confirms that Conjecture 1 implies Conjecture \ref{conj2}. \\
The proof of the proposition is complete.
\end{proof}   

\section[Study of related sums]{Study of the sums $\sum_{a \in \A} \frac{1}{a (\log a + x)}$}
In this section, we study (for a given $x \geq 0$) the series $\sum_{a \in \A} \frac{1}{a (\log a + x)}$, where $\A$ runs on the set of all primitive sequences different from $\{1\}$. Although our first objective is to disprove (for some $x$'s) the analogue of the \erdos{} conjecture related to those sums, we will prove the following stronger result:     
\begin{thm}\label{t1}
For every $\lambda \geq 1$ and every $x \geq 2310 \, \lambda \big(\log(\lambda + 2)\big)^{5/2}$, there exists a primitive sequence $\A \neq \{1\}$ (effectively constructible), satisfying the inequality:
$$
\sum_{a \in \A} \frac{1}{a (\log a + x)} > \lambda \sum_{p \in \P} \frac{1}{p (\log p + x)} .
$$
\end{thm}
To prove this theorem, we need effective estimates of the $n$\textsuperscript{th} prime number $p_n$ in terms of $n$ together with an effective lower bound of the sum $\sum_{p \in \P , p \leq x} \frac{1}{p}$ ($x > 1$) in terms of $x$. According to \cite[Theorem A, items (i) and (iv)]{mr}, we have:
\begin{eqnarray}
p_n & \geq & n \log{n} ~~~~~~~~~~ (\forall n \geq 2) \label{eq3} \\
p_n & \leq & n \left(\log n + \log\log n\right) ~~~~~~~~~~ (\forall n \geq 6) \label{eq4}
\end{eqnarray} 
On the other hand, we can check by hand that we have $p_n \leq n^2$ for $2 \leq n < 6$. From this fact, together with \eqref{eq3} and \eqref{eq4}, it follows that for any $n \geq 2$, we have:
\begin{equation}\label{eq5}
\log n \leq \log p_n \leq 2 \log n
\end{equation}
Next, according to \cite[Estimate (3.19), page 70]{rs}, we have for any $x > 1$:
\begin{equation}\label{eq-a}
\sum_{\begin{subarray}{c}
p \in \P \\
p \leq x
\end{subarray}} \frac{1}{p} > \log\log x
\end{equation}
Furthermore, we need the two following lemmas:
\begin{lemma}\label{l1}
For any positive real number $x$ and any positive integer $k \geq 2$, we have:
$$
\sum_{n > k} \frac{1}{p_n \left(\log{p_n} + x\right)} \leq \frac{\log\left(1 + \frac{x}{\log k}\right)}{x} .
$$
\end{lemma}
\begin{proof}
Let $x$ be a positive real number and $k \geq 2$ be an integer. According to \eqref{eq3} and \eqref{eq5}, we have:
\begin{eqnarray*}
\sum_{n > k} \frac{1}{p_n (\log p_n + x)} & \leq & \sum_{n > k} \frac{1}{n \log n (\log n + x)} \\
& < & \int_{k}^{+ \infty} \frac{d t}{t \log t (\log t + x)}
\end{eqnarray*} 
(since the function $t \mapsto \frac{1}{t \log t (\log t + x)}$ decreases on the interval $(1 , + \infty)$). \\
Next, we have:
\begin{eqnarray*}
\int_{k}^{+ \infty} \frac{d t}{t \log t (\log t + x)} & = & \int_{\log k}^{+ \infty} \frac{d u}{ u (u + x)} ~~~~~~~~~~ \text{(by setting $u = \log t$)} \\
& = & \frac{1}{x} \int_{\log k}^{+ \infty} \left(\frac{1}{u} - \frac{1}{u + x}\right) \, d u \\
& = & \frac{1}{x} {\Big[\log u - \log(u + x)\Big]}_{u = \log k}^{+ \infty} \\
& = & \frac{1}{x} \log\left(1 + \frac{x}{\log k}\right) .
\end{eqnarray*}
The inequality of the lemma then follows.
\end{proof}
\begin{lemma}\label{l2}
For any positive integer $n$, we have:
$$
n! \leq n^n e^{1 - n} \sqrt{n} .
$$
\end{lemma}
\begin{proof}
For $n = 1$, the inequality of the lemma clearly holds. For $n \geq 2$, it is an
immediate consequence of the more precise well-known inequality
$n! \leq n^n e^{-n} \sqrt{2 \pi n} \, e^{1 / 12 n}$, which can be found in
Problem 1.15 of \cite{kf}.
\end{proof}

Now, we are ready to prove Theorem \ref{t1}.

\begin{proof}[Proof of Theorem \ref{t1}]
Let us fix $\lambda \geq 1$. We introduce three positive parameters $c , \alpha$ and $\beta$ (independent from $\lambda$) which must satisfy the conditions
\begin{align}
c \alpha & \geq e^{\beta} + \log 2 \tag{$C_1$} \label{eq-h1} \\
\beta & \geq \frac{5}{2} \tag{$C_2$} \label{eq-h2}
\end{align}
Those parameters will be chosen at the end to optimize our result. We must choose $c$ to be the smallest possible value such that for any $x \geq c \lambda \left(\log(\lambda + 2)\right)^{5/2}$, there exists a primitive sequence $\A \neq \{1\}$, satisfying $\sum_{a \in \A} \frac{1}{a (\log{a} + x)} > \lambda \sum_{p \in \P} \frac{1}{p (\log{p} + x)}$. Two other parameters $k$ and $d$ are considered; they are both positive integers depending on $\lambda$. Especially, the choice of $d$ in terms of $\lambda$ can be easily understood towards the end of the proof. \\
Let $x \geq c \lambda \big(\log(\lambda + 2)\big)^{5/2}$. We define $p_k$ as the greatest prime number satisfying $p_k \leq e^{\alpha x}$. So, we have:
\begin{equation}\label{eq-b}
p_k \leq e^{\alpha x} < p_{k + 1} < 2 p_k
\end{equation}
(where the last inequality is a consequence of Bertrand's postulate). Note that \eqref{eq-h1} insures that $k \geq 2$. Hence (using \eqref{eq5}):
$$
\log p_k \leq \alpha x < \log p_k + \log 2 \leq 2 \log k + \log 2 \leq 3 \log k ,
$$
that is
\begin{equation}\label{eq8}
\log p_k \leq \alpha x < 3 \log k
\end{equation}
Next, set $d := \lfloor\log{\lambda} + \frac{5}{2} \log\log(\lambda + 2) + \beta\rfloor$. By using successively \eqref{eq-a}, Bertrand's postulate, \eqref{eq-b}, and the estimate $x > \frac{1}{\alpha} \left(e^d + \log 2\right)$ (resulting from \eqref{eq-h1}), we get:
$$
\sum_{n = 1}^{k} \frac{1}{p_n} > \log\log p_k > \log\log\left(\frac{p_{k + 1}}{2}\right) > \log\log\left(\frac{e^{\alpha x}}{2}\right) > d ,
$$ 
that is
\begin{equation}\label{eq6}
\sum_{n = 1}^{k} \frac{1}{p_n} > d
\end{equation}
Then, by using successively \eqref{eq8} and \eqref{eq6}, we get:
$$
\sum_{n = 1}^{k} \frac{1}{p_n (\log p_n + x)} \geq \frac{1}{(1 + \alpha) x} \sum_{n = 1}^{k} \frac{1}{p_n} >  \frac{d}{(1 + \alpha) x} .
$$
On the other hand, by using successively Lemma \ref{l1} and \eqref{eq8}, we get:
$$
\sum_{n > k} \frac{1}{p_n (\log p_n + x)} \leq \frac{1}{x} \log\left(1 + \frac{x}{\log k}\right) < \frac{1}{x} \log\left(1 + \frac{3}{\alpha}\right) .
$$
By comparing the two last estimates, we obviously deduce that:
$$
\sum_{n = 1}^{k} \frac{1}{p_n (\log p_n + x)} > \frac{d}{(1 + \alpha) \log\left(1 + \frac{3}{\alpha}\right)} \sum_{n > k} \frac{1}{p_n (\log p_n + x)} .
$$
By adding to both sides of this inequality the quantity $\frac{d}{(1 + \alpha) \log\left(1 + \frac{3}{\alpha}\right)} \sum_{n = 1}^{k} \frac{1}{p_n (\log p_n + x)}$, we deduce (after simplifying) that:
\begin{equation}\label{eq9}
\sum_{n = 1}^{k} \frac{1}{p_n (\log p_n + x)} > \frac{d}{d + (1 + \alpha) \log\left(1 + \frac{3}{\alpha}\right)} \sum_{n = 1}^{+ \infty} \frac{1}{p_n (\log p_n + x)}
\end{equation}
Now, let $\A$ be the set of positive integers defined by:
$$
\A := \left\{p_1^{\alpha_1} p_2^{\alpha_2} \cdots p_k^{\alpha_k} |~ \alpha_1 , \dots , \alpha_k \in \N , \alpha_1 + \dots + \alpha_k = d\right\} .
$$
Since $\A$ is homogeneous (of degree $d$) then it is a primitive set. For a suitable choice of $c , \alpha$ and $\beta$, we will show that $\A$ satisfies the inequality of the theorem. We have:
\begin{eqnarray*}
\sum_{a \in \A} \frac{1}{a} & = & \sum_{\alpha_1 + \dots + \alpha_k = d} \frac{1}{p_1^{\alpha_1} p_2^{\alpha_2} \cdots p_k^{\alpha_k}} \\
& \geq & \sum_{\alpha_1 + \dots + \alpha_k = d} \frac{(1/p_1)^{\alpha_1}}{\alpha_1 !} \frac{(1/p_2)^{\alpha_2}}{\alpha_2 !} \cdots \frac{(1/p_k)^{\alpha_k}}{\alpha_k !} \\
& = & \frac{1}{d!} \left(\sum_{i = 1}^{k} \frac{1}{p_i}\right)^d ~~~~~~~~~~ \text{(according to the multinomial formula)} \\
& = & \frac{1}{d!} \left(\sum_{i = 1}^{k} \frac{1}{p_i}\right)^{d - 1} \left(\sum_{i = 1}^{k} \frac{1}{p_i}\right) \\
& > & \frac{d^{d - 1}}{d!} \sum_{i = 1}^{k} \frac{1}{p_i} ~~~~~~~~~~~~~~ \text{(according to \eqref{eq6})} ,
\end{eqnarray*}
that is
\begin{equation}\label{eq10}
\sum_{a \in \A} \frac{1}{a} > \frac{d^{d - 1}}{d!} \sum_{n = 1}^{k} \frac{1}{p_n}
\end{equation}
Further, since $p_k^d$ is obviously the greatest element of $\A$, we have for any
$a \in \A$: $\log a \leq \log(p_k^d) = d \log p_k \leq d \alpha x$ (according to
\eqref{eq8}), that is
\begin{equation}\label{eq11}
\log a \leq d \alpha x ~~~~~~~~~~ (\forall a \in \A)  
\end{equation}
By combining the above estimates, we get
\begin{eqnarray*}
\sum_{a \in \A} \frac{1}{a (\log a + x)} & \geq & \frac{1}{(d \alpha + 1) x}
\sum_{a \in \A} \frac{1}{a} ~~~~~~~~~~~~~~~~~~~~~~~~~~~~~~
\text{(according to \eqref{eq11})} \\
& > & \frac{1}{(d \alpha + 1) x} \cdot \frac{d^{d - 1}}{d!} \sum_{n = 1}^{k} \frac{1}{p_n} 
~~~~~~~~~~~~~~~~~~~~~~
\text{(according to \eqref{eq10})} \\
& > & \frac{d^{d - 1}}{d! (d \alpha + 1)} \sum_{n = 1}^{k} \frac{1}{p_n (\log{p_n} + x)} 
\end{eqnarray*}~\vspace*{-7mm}
\begin{eqnarray*}
& > & \frac{d^{d - 1}}{d! (d \alpha + 1)} \cdot \frac{d}{d + (1 + \alpha) \log\left(1 + \frac{3}{\alpha}\right)} \sum_{n = 1}^{+ \infty}
\frac{1}{p_n (\log{p_n} + x)} ~~~~~~~~~~ \text{(according to \eqref{eq9})} \\
& = & \frac{d^d}{d!} \cdot \frac{1}{(d \alpha + 1) \left(d + (1 + \alpha) \log\left(1 + \frac{3}{\alpha}\right)\right)} \sum_{n = 1}^{+ \infty}
\frac{1}{p_n (\log{p_n} + x)}
\end{eqnarray*}

\noindent Then, using Lemma \ref{l2}, it follows that:
\begin{equation}\label{eq-c}
\sum_{a \in \A} \frac{1}{a (\log a + x)} > \frac{e^{d - 1}}{\sqrt{d} (d \alpha + 1) \left(d + (1 + \alpha) \log\left(1 + \frac{3}{\alpha}\right)\right)} \sum_{n = 1}^{+ \infty} \frac{1}{p_n (\log{p_n} + x)}
\end{equation}
But according to the expression of $d$ in terms of $\lambda$, we have clearly:
\begin{equation}\label{eq-prov1}
e^{d - 1} > e^{\beta - 2} \lambda \left(\log(\lambda + 2)\right)^{5/2}
\end{equation}
and by using in addition the obvious estimates $\log{\lambda} < \log(\lambda + 2)$, $\log\log(\lambda + 2) \leq \log(\lambda + 2) - 1$ and the condition \eqref{eq-h2}, we have:
$$
d < (\beta + 1) \log(\lambda + 2),
$$
which implies the following:
\begin{eqnarray}
\sqrt{d} & < & \sqrt{\beta + 1} \left(\log(\lambda + 2)\right)^{1/2} \label{eq-prov2} \\
d \alpha + 1 & < & \left((\beta + 1) \alpha + 1\right) \log(\lambda + 2) \label{eq-prov3} \\
d + (1 + \alpha) \log\left(1 + \frac{3}{\alpha}\right) & < & \left(\beta + 1 + (1 + \alpha) \log\left(1 + \frac{3}{\alpha}\right)\right) \log(\lambda + 2) \label{eq-prov4}
\end{eqnarray}
By inserting all the estimates \eqref{eq-prov1}, \eqref{eq-prov2}, \eqref{eq-prov3} and \eqref{eq-prov4} into \eqref{eq-c}, we finally obtain:
\begin{equation}\label{eq-d}
\sum_{a \in \A} \frac{1}{a (\log a + x)} > \frac{e^{\beta - 2} \lambda}{\sqrt{\beta + 1} \big((\beta + 1) \alpha + 1\big) \left(\beta + 1 + (1 + \alpha) \log\left(1 + \frac{3}{\alpha}\right)\right)} \sum_{n = 1}^{+ \infty} \frac{1}{p_n (\log{p_n} + x)}
\end{equation}
To obtain the required inequality of the theorem, we must choose $\alpha$ and $\beta$ (with $\alpha > 0$ and $\beta \geq \frac{5}{2}$) such that:
\begin{equation}
\frac{e^{\beta - 2}}{\sqrt{\beta + 1} \big((\beta + 1) \alpha + 1\big) \left(\beta + 1 + (1 + \alpha) \log\left(1 + \frac{3}{\alpha}\right)\right)} \geq 1 \tag{$\star$}
\end{equation}
After that, we can take (according to \eqref{eq-h1}): $c = \frac{e^{\beta} + \log{2}}{\alpha}$. So, to obtain an optimal result, we should choose $\alpha$ and $\beta$ so that $(\star)$ holds and $\frac{e^{\beta} + \log{2}}{\alpha}$ is the smallest possible. Using Excel's solver, we find the solution $(\alpha , \beta) = (0.44516\dots , 6.93492\dots)$, which gives $c \simeq 2309.8$ and concludes this proof.
\end{proof}

\begin{rmq}\label{rmq1}
For the more significant case $\lambda = 1$, it is possible to improve the result of Theorem \ref{t1} by ignoring in the preceding proof the parameter $\beta$ and working directly with $d$. Doing so, the optimization problem that we have to solve consists of minimizing the quantity $\frac{1}{\alpha} (e^d + \log{2})$ under the constraints $\frac{e^{d - 1}}{\sqrt{d} (d \alpha + 1) (d + (1 + \alpha) \log(1 + 3/{\alpha}))} \geq 1$ and $d \in \mathbb{Z}^+$. We obtain the following:
\end{rmq}

\begin{thm}\label{t2}
For every real number $x \geq 363$, there exists a primitive sequence $\A \neq \{1\}$ (effectively constructible), satisfying the inequality:
$$
\sum_{a \in \A} \frac{1}{a (\log a + x)} > \sum_{p \in \P} \frac{1}{p (\log p + x)} .
$$
\end{thm}
\begin{proof}
We introduce the positive parameters $c , \alpha , d$ and $k$ such that $d$ and $k$ are integers and
\begin{equation}
c \alpha \geq e^d + \log{2} \tag*{$(C_1)'$}
\end{equation}
We shall choose $c$ to be the smallest possible value such that for any $x \geq c$, there exists a primitive sequence $\A \neq \{1\}$, satisfying $\sum_{a \in \A} \frac{1}{a (\log{a} + x)} > \sum_{p \in \P} \frac{1}{p (\log{p} + x)}$. By taking $k$ and $\A$ as in the preceding proof of Theorem \ref{t1} and by reproducing the same arguments of that proof, we arrive at the inequality:
\begin{equation}\label{eq-e}
\sum_{a \in \A} \frac{1}{a (\log a + x)} > \frac{e^{d - 1}}{\sqrt{d} (d \alpha + 1) \left(d + (1 + \alpha) \log\left(1 + \frac{3}{\alpha}\right)\right)} \sum_{n = 1}^{+ \infty} \frac{1}{p_n (\log{p_n} + x)} \tag*{(\ref{eq-c})$'$}
\end{equation} 
To obtain an optimal result, we should choose $d$ and $\alpha$ such that
\begin{equation}
\frac{e^{d - 1}}{\sqrt{d} (d \alpha + 1) \left(d + (1 + \alpha) \log\left(1 + \frac{3}{\alpha}\right)\right)} \geq 1 \tag*{$(\star)'$}
\end{equation}
and $\frac{e^d + \log{2}}{\alpha}$ is minimal. Using Excel's solver, we find the solution $(d , \alpha) = (5 , 0.41154\dots)$, which gives $\frac{e^d + \log{2}}{\alpha} \simeq 362.313$. So, to satisfy $(C_1)'$, we can take $c = 363$. This completes the proof.
\end{proof}

For any natural number $k$, let us define the homogeneous sequence:
$$
\P_k := \left\{n \in \Z^+ :~  \Omega(n) = k\right\} .
$$
In particular, we have $\P_0 = \{1\}$ and $\P_1 = \P$. According to the proof of Theorem \ref{t2}, a concrete example of primitive sequence $\A \neq \{1\}$ which satisfies the inequality $\sum_{a \in \A} \frac{1}{a (\log{a} + x)} > \sum_{p \in \P} \frac{1}{p (\log{p} + x)}$ (for any $x$ sufficiently large) is $\A = \P_5$. Next, Theorem \ref{t2} shows that for $x \geq 363$, the sum $\sum_{a \in \A} \frac{1}{a (\log a + x)}$ (where $\A$ runs on the set of all primitive sequences different from $\{1\}$) does not reach its maximum value at $\A = \P$. This shows that the analogue of the \erdos{} conjecture for the sums $\sum_{a \in \A} \frac{1}{a (\log a + x)}$ is in general false. So, it is natural to ask, for a given $x \in [0 , + \infty)$, if the supremum of the preceding sum is attained and, if so, what is the structure of a maximizing primitive sequence $\A$. A conjectural answer of this question is proposed by the following conjecture, generalizing the \erdos{} one while remaining more vague:
\begin{conj}\label{conj3}
For any non-negative real number $x$, the sum 
$$
\sum_{a \in \A} \frac{1}{a (\log a + x)}
$$
(where $\A$ runs on the set of all primitive sequences different from $\{1\}$) reaches its maximum value at some primitive sequence of the form $\P_k$ ($k \geq 1$).
\end{conj} 

\begin{rmq}
According to the result of Zhang \cite{z}, showing that \erdos's conjecture is true for the particular case of homogeneous sequences, our conjecture \ref{conj3} immediately implies that of \erdos{}.  
\end{rmq}

\subsection*{Some other open questions}
Let $\I$ be the set of the non-negative real numbers $x$, satisfying the property that
$$
\sup_{\A} \sum_{a \in \A} \frac{1}{a (\log a + x)} > \sum_{p \in \P} \frac{1}{p
(\log p + x)} ,
$$
where in the left-hand side of this inequality, the supremum is taken over all primitive sequences $\A \neq \{1\}$. Then, \erdos's conjecture can be reformulated just by saying that $0 \not\in \I$. Further, Theorem \ref{t2} shows that $\I \supset [363 , + \infty)$. That said, several other informations concerning $\I$ remain unknown; we can ask for example the following questions:
\begin{enumerate}
\item[(1)] Is $\I$ an interval? Is it an open set of $\R$? Is it a closed set of $\R$? (where, in the two last questions, $\R$ is equipped with its usual topology).
\item[(2)] Determine the infimum of $\I$ (i.e., $\inf \I$).
\end{enumerate}

\end{document}